\documentclass{amsart}
\usepackage[utf8]{inputenc}
\usepackage[utf8]{inputenc}
\usepackage[T1]{fontenc}
\usepackage[all]{xy}
\usepackage{lipsum}
\usepackage{url}
\usepackage{stackrel}

\usepackage{amsmath,amsthm,amssymb}

\usepackage{graphicx}

\theoremstyle{definition}
\newtheorem{definition}{Definition}[section]

\newtheorem{theorem}{Theorem}[section]

\newtheorem{example}[theorem]{Example}
\newtheorem{proposition}[theorem]{Proposition}
\newtheorem{lemma}[theorem]{Lemma}
\newtheorem{remark}[theorem]{Remark}

\numberwithin{equation}{section}

\begin{document}


\vspace{0.5in}

\renewcommand{\bf}{\bfseries}
\renewcommand{\sc}{\scshape}
\vspace{0.5in}

\title[Category and Topological Complexity]%
{Category and Topological Complexity of the configuration space $F(G\times \mathbb{R}^n,2)$ \\ }

\author{Cesar A. Ipanaque Zapata}
\address{Deparatmento de Matem\'{a}tica,UNIVERSIDADE DE S\~{A}O PAULO
INSTITUTO DE CI\^{E}NCIAS MATEM\'{A}TICAS E DE COMPUTA\c{C}\~{A}O -
USP , Avenida Trabalhador S\~{a}o-carlense, 400 - Centro CEP:
13566-590 - S\~{a}o Carlos - SP, Brasil}
\email{cesarzapata@usp.br}



\thanks {The author wishes to acknowledge support for this research from grant\#2016/18714-8, São Paulo Research Foundation (FAPESP). Also, the author is very grateful to Oziride M. Neto for their comments and encouraging remarks which were of invaluable.}

\begin{abstract} The Lusternik-Schnirelmann category cat and topological complexity TC are related homotopy invariants. The topological complexity TC has applications to the robot motion planning problem. We calculate the Lusternik-Schnirelmann category and topological complexity of the ordered configuration space of two distinct points in the product $G\times\mathbb{R}^n$ and apply the results to the planar and spatial motion of two rigid bodies in $\mathbb{R}^2$ and $\mathbb{R}^3$ respectively.
\end{abstract}

\maketitle


\section{\bf Introduction}

Let $X$ be the space of all possible configurations or states of a mechanical system. A motion planning algorithm on $X$ is a function which assigns to any pair of configurations $(A,B)\in X\times X$, an initial state $A$ and a desired state $B$, a continuous motion of the system starting at the initial state $A$ and ending at the desired state $B$. The elementary problem of robotics, \textit{the motion planning problem}, consists of finding a motion planning algorithm for a given mechanical system. The motion planning algorithm should be continuous, that is, it depends continuously on the pair of points $(A,B)$. Absence of continuity will result in instability of the behavior of the motion planning. Unfortunately, a (global) continuous motion planning algorithm on a space $X$ exists if and only if $X$ is contractible (see \cite{farber2003topological}). If $X$ is not contractible, then only local continuous motion plans may be found. Informally, the topological complexity, TC$(X)$ is the minimal number of local  continuous motion plans, \textit{effective motion planning algorithms}, which are needed to construct an algorithm for autonomous motion planning of a system having $X$ as its state space. The design of effective motion planning algorithms is one of the challenges of modern robotics (see, for example Latombe \cite{latombe2012robot} and LaValle \cite{lavalle2006planning}).

Investigation of the problem of simultaneous motion planning without collisions for $k$ robots in a topological space $X$ leads one to study \textit{the ordered configuration space} $F(X,k)$ of $k$ distinct points of a topological space $X$ (see \cite{fadell1962configuration}). It is the subset   \[F(X,k)=\{(x_1,\ldots,x_k)\in X^k\mid ~~x_i\neq x_j\text{ for } i\neq j \},\] topologised as a subspace of the Cartesian power $X^k$. This space is used in robotics to control multiple objects simultaneously, trying to avoid collisions between them \cite{farber2008invitation}. 
 
In \cite{zapata2017non}, the author shows that the ordered configuration spaces $F(M,k)$ of topological manifolds $M$ are never contractible. Thus, the collision-free  simultaneous motion planning problem on a manifold is a major challenge. Indeed, computation of Lusternik-Schnirelmann category LS and the topological complexity TC of the configuration space $F(M,k)$ is very difficult. The LS category of the configuration space  $F(\mathbb{R}^m,k)$ has been computed by Roth in \cite{roth2008category} and  TC$(F(\mathbb{R}^m,k))$ for $m=2$ and $m$ odd was computed by Farber and Yuzvinsky in  \cite{farber2004topological}. Farber and Grant \cite{farber2009topological}, extended the results to all dimensions $m$.  Farber et al. \cite{farber2007topological} determined the topological complexity of $F(\mathbb{R}^m-Q_r,k)$ for $m=2,3$. Later Gonz\'{a}lez and Grant \cite{gonzalez2015sequential} extended the results to all dimensions $m$. Cohen and Farber \cite{cohen2011topological} computed the topological complexity of the configuration space $F(\Sigma_g-Q_r,k)$ of orientable surfaces $\Sigma_g$. Recently in \cite{zapata2017lusternik}, the author computed the LS category and TC of the configuration space $F(\mathbb{CP}^m,2)$. Many more related results can be found in the recent survey papers \cite{cohen2018topological} and \cite{farber2017configuration}. 

In this paper we calculate the Lusternik-Schnirelmann category and topological complexity of the ordered configuration space of two distinct points in the product $G\times\mathbb{R}^n$, where $G$ is a compact connected Lie group satisfying certain conditions and $n$ is a natural number (see Theorem \ref{theorem-cat} and Theorem \ref{theorem=tc}, respectively).

\section{Preliminary results}

 \begin{remark}\label{configu-group}
Let $k\geq 1$.  If $G$ is a topological group, then it is well-known that the configuration space $F(G,k+1)$ is homeomorphic to the product \[G\times F(G-\{e\},k),\] under the homeomorphism $(g_0,g_1,\ldots,g_k)\mapsto (g_0,g_1g_0^{-1},\ldots,g_kg_0^{-1})$ with its inverse $(g_0,g_1,\ldots,g_k)\mapsto (g_0,g_1g_0,\ldots,g_kg_0)$.  Here $e$ denotes the identity element of the group $G$.
 \end{remark}
 
\begin{definition}
The \textit{Lusternik-Schnirelmann category} (LS category) or category of a topological space $X$, denoted cat$(X)$, is the least integer $m$ such that $X$ can be covered with $m$ open sets, which are all contractible within $X$. 
\end{definition}

We use a definition of category which is one greater than that given in \cite{cornea2003lusternik}. For example, the category of a contractible space is one.

\begin{example}\label{cat-wedges-spheres}%

If $Z=\underbrace{\mathbb{S}^{m_1}\vee\cdots\vee\mathbb{S}^{m_n}}_{n \text{ factors }}$ is a wedge of spheres $\mathbb{S}^{m_i}$, then
\begin{equation*}
\text{cat}(Z)=2.
\end{equation*}
\end{example}

Farber \cite{farber2003topological} defined a numerical invariant TC$(X)$. Let $PX$ denote the space of all continuous paths $\gamma: [0,1] \longrightarrow X$ in $X$ and  $\pi: PX \longrightarrow X \times X$ denote the map associating to any path $\gamma\in PX$ the pair of its initial and end points $\pi(\gamma)=(\gamma(0),\gamma(1))$. Equip the path space $PX$ with the compact-open topology. 

\begin{definition}\cite{farber2003topological}
The \textit{topological complexity} of a path-connected space $X$, denoted by TC$(X)$, is the least integer $m$ such that the Cartesian product $X\times X$ can be covered with $m$ open subsets $U_i$, \begin{equation*}
        X \times X = U_1 \cup U_2 \cup\cdots \cup U_m, 
    \end{equation*} such that for any $i = 1, 2, \ldots , m$ there exists a continuous local section $s_i : U_i \longrightarrow PX$ of $\pi$, that is, $\pi\circ s_i = id$ over $U_i$. If no such $m$ exists we will set TC$(X)=\infty$. 
\end{definition}

\begin{remark}
 We recall that TC$(G)=$ cat$(G)$ for any connected Lie group $G$ (see \cite{farber2004instabilities}, Lemma 8.2). 
\end{remark}

Next we give the definition of monoidal topological complexity, again one greater than that given in \cite{dranishnikov2014topological}.

\begin{definition}\cite{dranishnikov2014topological}
The \textit{monoidal topological complexity} of a path-connected space $X$, denoted by TC$^M(X)$, is the least integer $m$ such that the Cartesian product $X\times X$ can be covered with $m$ open subsets $U_i$, \begin{equation*}
        X \times X = U_1 \cup U_2 \cup\cdots \cup U_m, 
    \end{equation*} such that for any $i = 1, 2, \ldots , m$ there exists a continuous local section $s_i : U_i \longrightarrow PX$ of $\pi$, that is, $\pi\circ s_i = id$ over $U_i$, and if $(x,x)\in U_i$ then $s_i(x,x)(t)=x,\forall t\in [0,1]$. If no such $m$ exists, we set TC$^M(X)=\infty$. The motion planning algorithm $s$ is called the \textit{reserved motion planning algorithm}.
\end{definition}

\begin{remark}
 One of the basic properties of cat$(X)$ and TC$(X)$ is their homotopy invariance (\cite{cornea2003lusternik}, Theorem 1.30; \cite{farber2003topological}, Theorem 3). In contrast, TC$^M$ is not a homotopy invariant in general (see \cite{iwase2010topological}). From (\cite{dranishnikov2014topological}, Theorem 2.1), if $X$ is a finite CW complex, TC$(X)\leq $ TC$^M(X)\leq$  TC$(X)+1$ .
\end{remark}

\begin{example}
 TC$(\mathbb{S}^m)=$ TC$^M(\mathbb{S}^m)$ for any $m\geq 1$ (\cite{dranishnikov2014topological}, Corollary 2.6). Furthermore, 
 TC$(G)=$ TC$^M(G)$ for any connected Lie group $G$ (\cite{dranishnikov2014topological}, Lemma 2.7).
\end{example}

Let $\mathbb{K}$ be a field. The singular cohomology $H^{*}(X;\mathbb{K}):=H^{*}(X)$ is a graded $\mathbb{K}-$algebra with multiplication \[\cup:H^{*}(X)\otimes H^{*}(X)\longrightarrow H^{*}(X)\] given by the cup-product. The tensor product $H^{*}(X)\otimes H^{*}(X)$ is also a graded $\mathbb{K}-$algebra with the multiplication
\[(u_1\otimes v_1)\cdot (u_2\otimes v_2):=(-1)^{\deg(v_1)\deg(u_2)}u_1u_2\otimes v_1v_2,\]
 where $\deg(v_1)$ and $\deg(u_2)$ denote the degrees of cohomology classes $v_1$ and $u_2$ respectively. The cup-product $\cup$ is a homomorphism of $\mathbb{K}-$algebras. 
 
 \begin{definition}(\cite{farber2003topological}, Definition 6)
 The kernel of homomorphism $\cup$ is \textit{the ideal of the zero-divisors} of $H^{*}(X)$. The \textit{zero-divisors-cup-length} of $H^{*}(X)$, denoted $zcl(H^{*}(X))$, is the length of the longest nontrivial product in the ideal of the zero-divisors of $H^{*}(X)$.
 \end{definition}

Proposition \ref{prop-1} below gives the general properties of the category and topological complexity of a space $X$:

\begin{proposition}\label{prop-1} 
\begin{enumerate}
    \item  (\cite{cornea2003lusternik}, Theorem 1.5)
If $R$ is a commutative ring with unit and $X$ is a topological space, then \[1+cup_R(X)\leq \text{cat}(X),\] where  $cup_R(X)$ is the least integer $n$ such that all $(n+1)-$fold cup products vanish in the reduced cohomology $\widetilde{H^\star}(X;R)$.

\item If $\mathbb{K}$ is a field and $X$ be a path-connected topological space, then \[1+zcl_{\mathbb{K}}(X)\leq \text{TC}(X).\] 

\end{enumerate}

\end{proposition}

It is easy to verify that the cup-length and the zero-divisor cup-length have the properties listed below.

\begin{lemma}\label{prop-cup-length}
Let $\mathbb{K}$ be a field and $X,Y$ be topological spaces. Then
\begin{enumerate}
\item If $H^k(Y;\mathbb{K})$ is a finite dimensional $\mathbb{K}-$vector space for all $k\geq 0$, then \begin{equation*}cup_{\mathbb{K}}(X\times Y)= cup_{\mathbb{K}}(X)+cup_{\mathbb{K}}(Y);\end{equation*}
\item If $X,Y$ is CW complexes, then \begin{equation*} cup_{\mathbb{K}}(X\vee Y)= \max\{cup_{\mathbb{K}}(X),cup_{\mathbb{K}}(Y)\}.\end{equation*} Furthermore, \begin{equation*} zcl_{\mathbb{K}}(X\vee Y)\geq \max\{zcl_{\mathbb{K}}(X),zcl_{\mathbb{K}}(Y)\}.\end{equation*}
\end{enumerate}
\end{lemma}


\section{\bf Main Results}

We first recall some lemmas. We denote by $Int(M)$ the interior of the manifold $M$.

 \begin{lemma}\label{m-d-igual-m-p}
 Let $M$ be a connected $m-$dimensional smooth manifold (with or without boundary). Let $D_1,\ldots,D_k\subseteq Int(M)$ be subsets homeomorphic to an $m-$dimensional ball $\mathbb{D}^m=\{x\in \mathbb{R}^m\mid~~\parallel x\parallel\leq 1\}$ such that each $D_i$ has a neighbourhood $V_i\subseteq Int(M)$, where $D_i\subsetneq V_i$ and $V_i$ is also homeomorphic to $\mathbb{D}^m$ and $V_i\cap V_j=\emptyset$ for $~i\neq j$. Let $p_1,\ldots,p_k$ the centers of $D_1,\ldots,D_k$, respectively. Then the complement $ M-\bigcup_{i=1}^{k}D_i$ is homeomorphic to $M-\{p_1,\ldots,p_k\}$.
 \end{lemma}
 
 For $m\geq 0$, let $Q_m\subseteq Int(M)$ be a finite subset of $m$ points in $Int(M)$.
 
 \begin{lemma}\label{inavarianca-bordo}\cite{zapata2017collision}
Let $M$ be a connected and compact smooth manifold with nonempty boundary. Then for each $k\geq 1$, the inclusion map $i:Int(M)-Q_m\hookrightarrow M-Q_m$ induces homotopy equivalences in the configuration space $F(M-Q_m,k)$, that is, the map \begin{equation*}
 \begin{array}{rccl}
 F(i,k):&F(Int(M)-Q_m,k)&\longrightarrow & F(M-Q_m,k)\\&(x_1,\ldots,x_k)&\mapsto
 &(x_1,\ldots,x_k)
                           \end{array}
 \end{equation*} is a homotopy equivalence. Furthermore, there is a $\Sigma_k-$equivariant deformation retraction of $F(M-Q_m,k)$ onto $F(Int(M)-Q_m,k)$. 
 \end{lemma}
 
 \begin{lemma}\label{m-d}
 Let $M$ be a connected $m-$dimensional smooth manifold having a nonempty boundary $\partial M$. Let $D\subseteq M$ be a subset homeomorphic to an $m-$dimensional ball $\mathbb{D}^m=\{x\in \mathbb{R}^m\mid~~\parallel x\parallel\leq 1\}$, lying in the interior of $M$ and such that the boundary $\partial D$ is piecewise smooth. Then the complement $M-D$ is homotopy equivalent to the wedge $M\vee \mathbb{S}^{m-1}$.
 \end{lemma}
 
 \begin{remark}
 Lemma \ref{m-d} is not true if $\partial M=\emptyset$ (for example, if $M=\mathbb{S}^1\times \mathbb{S}^1$ the torus).
 \end{remark}
 
 \begin{lemma}\label{mtimesr-p}
  Let $m,n\geq 1$. If $M$ is a compact, connected $m-$dimensional smooth manifold without boundary and $x_0\in M\times\mathbb{R}^n$, then the complement $M\times\mathbb{R}^n-\{x_0\}$ is homotopy equivalent to the wedge $M\vee \mathbb{S}^{m+n-1}$.
 \end{lemma}
 \begin{proof}
By Lemma \ref{inavarianca-bordo}, the complement  $M\times\mathbb{R}^n -\{x_0\}$ is homotopy equivalent to the complement $M\times\mathbb{D}^n -\{x_0\}$ (here we recall that $x_0\in Int(M\times\mathbb{D}^n)$). Let $D\subseteq M\times\mathbb{D}^n$ be a subset homeomorphic to an $(m+n)-$dimensional ball $\mathbb{D}^{m+n}=\{x\in \mathbb{R}^{m+n}\mid~~\parallel x\parallel\leq 1\}$, lying in the interior of $M\times\mathbb{D}^n$ and  such that the boundary $\partial D$ is piecewise smooth, $x_0\in D$ is the centre and $D$ has a neighbourhood $V\subseteq Int(M\times\mathbb{D}^n)$ with $D\subsetneq V$, where $V$ is also homeomorphic to $\mathbb{D}^{m+n}$. By Lemma \ref{m-d-igual-m-p} the complement $M\times\mathbb{D}^n -\{x_0\}$ is homeomorphic to $M\times\mathbb{D}^n-D$. Moreover, by Lemma \ref{m-d} the complement $M\times\mathbb{D}^n-D$ is homotopy equivalent to the wedge $(M\times\mathbb{D}^n)\vee \mathbb{S}^{m+n-1}$ and it is homotopy equivalent to $M\vee \mathbb{S}^{m+n-1}$. Therefore, the complement $M\times\mathbb{R}^n-\{x_0\}$ is homotopy equivalent to the wedge $M\vee \mathbb{S}^{m+n-1}$.
 \end{proof}
 
\begin{proposition}\label{homotopy-type}
Let $G$ be a connected, compact Lie group, $m=dim(G)\geq 1$ and $n\geq 1$. Then the configuration space $F(G\times\mathbb{R}^n,2)$ has the homotopy type of the product $G\times (G\vee \mathbb{S}^{m+n-1}):=\overline{F}(G\times\mathbb{R}^n,2)$.
\end{proposition} 
\begin{proof}
$F(G\times\mathbb{R}^n,2)$ is homotopy equivalent to $G\times (G\times\mathbb{R}^n-\{(e,0)\})$ (see Remark \ref{configu-group}). By Lemma \ref{mtimesr-p}, $F(G\times\mathbb{R}^n,2)$ is homotopy equivalent to $G\times (G\vee \mathbb{S}^{m+n-1})$.
\end{proof}

\begin{lemma}\label{cat-GxRn-2}
Let $\mathbb{K}$ be a field and  $G$ a connected, compact Lie group with $cup_{\mathbb{K}}(G)\geq 1$. Then 
\begin{equation*}
1+2cup_{\mathbb{K}}(G)\leq \text{cat}(F(G\times\mathbb{R}^n,2))\leq 2\text{cat}(G)-1, \text{ for any } n\geq 1.
\end{equation*}
\end{lemma}
\begin{proof}
 The category cat$(X)$ is a homotopy invariant of $X$, so by Proposition \ref{homotopy-type}, cat$(F(G\times\mathbb{R}^n,2))=$ cat$(G\times (G\vee \mathbb{S}^{m+n-1}))$. 
On the other hand, by (\cite{james1978category}, proposition 2.3), cat$(X\times Y)\leq$ cat$(X)$ $+$ cat$(Y)-1$. Thus \[cat(F(G\times\mathbb{R}^n,2))\leq \text{cat}(G) + \text{cat}(G\vee \mathbb{S}^{m+n-1})-1.\] Furthermore, cat$(X\vee Y)=\max\{\text{cat}(X),$ cat$(Y)\}$ and so \[\text{cat}(F(G\times\mathbb{R}^n,2))\leq 2\text{cat}(G)-1.\]

By Proposition \ref{prop-1}, cat$(F(G\times\mathbb{R}^n,2))\geq cup_{\mathbb{K}}(G\times (G\vee \mathbb{S}^{m+n-1}))+1$. Therefore, by Lemma \ref{prop-cup-length},  cat$(F(G\times\mathbb{R}^n,2))\geq 1+2cup_{\mathbb{K}}(G)$, as required. 
\end{proof}

\begin{remark} Let $G$ be a Lie group, as mentioned in Lemma \ref{cat-GxRn-2}. Then
  cat$(G)=cup_{\mathbb{K}}(G)+1$ if and only if  $1+2cup_{\mathbb{K}}(G)=$ cat$(F(G\times\mathbb{R}^n),2))=2$cat$(G)-1$.
\end{remark}

Thus, we have the following theorem.

\begin{theorem}\label{theorem-cat}
Let $\mathbb{K}$ be a field and let $G$ be a connected, compact Lie group such that $cup_{\mathbb{K}}(G)\geq 1$ and cat$(G)=cup_{\mathbb{K}}(G)+1$. Then 
\begin{equation*}
\text{cat}(F(G\times\mathbb{R}^n,2))= 2\text{cat}(G)-1, \text{ for any } n\geq 1. 
\end{equation*} Furthermore, cat$(F(G\times\mathbb{R}^n,2))=cup_{\mathbb{K}}(F(G\times\mathbb{R}^n,2))+1$.
\end{theorem}

\begin{example}
Since cat$(SO(m))=cup_{\mathbb{Z}_2}(SO(m))+1$, for $m=2,3,4,\ldots,10$ (see \cite{iwase2005lusternik}, \cite{iwase2016lusternik}), \begin{equation*} \text{cat}(F(SO(m)\times\mathbb{R}^n),2))=2\text{cat}(SO(m))-1, \text{ for } m=2,3,4,\ldots,10.\end{equation*}
\end{example}
 
\subsection{Topological complexity of wedges}

In general, there is no formula known to compute the topological complexity TC$(X\vee Y)$. However, Dranishnikov and Sadykov \cite{dranishnikov2017topological} proved the following statement. 

\begin{lemma}(\cite{dranishnikov2017topological}, Theorem 6)\label{dranishnikov-tc-wedges}
Let $d = \max\{dim~X, dim~Y\}$ for connected CW-complexes $X$ and $Y$. Suppose that $\max\{\text{TC}(X), \text{TC}(Y), \text{cat}(X\times Y)\}\geq d+2$. Then
\[\text{TC}(X\vee Y ) = \max\{\text{TC}(X), \text{TC}(Y), \text{cat}(X \times Y )\}.\]
\end{lemma}

Suppose $X=G$ is a connected noncontractible Lie group and $Y=\mathbb{S}^m$ are such that $dim(G)< m$ and cat$(G\times \mathbb{S}^m)=$ cat$(G)+1$. Then we cannot use Lemma \ref{dranishnikov-tc-wedges}, since  $\max\{\text{TC}(G), \text{TC}( \mathbb{S}^m), \text{cat}(G\times  \mathbb{S}^m)\}= \text{cat}(G)+1\leq dim(G)+2<m+2$. However, using Dranishnikov's method from (\cite{dranishnikov2014topological}, Theorem 3.6), we will prove the following result.

\begin{proposition}\label{tc-wedge}
Suppose that $X$ be a connected finite CW complex such that $\text{cat}(X)=\text{TC}(X)=\text{TC}^M(X)$ and $\text{cat}(X\times \mathbb{S}^m)=\text{cat}(X)+1$. (The last condition holds, for example, if $X$ satisfies Ganea's conjecture). Then
\[\text{TC}(X\vee \mathbb{S}^m)=\text{TC}(X)+1.\] Furthermore, $\text{TC}(X\vee \mathbb{S}^m)=\text{TC}^M(X\vee \mathbb{S}^m)=\max\{\text{TC}(X),\text{TC}(\mathbb{S}^m),\text{cat}(X\times \mathbb{S}^m)\}$.
\end{proposition}
\begin{proof}
It is well known that
$\text{cat}(X\times \mathbb{S}^m)\leq \text{TC}(X\vee \mathbb{S}^m)$ and so $\text{TC}(X)+1\leq \text{TC}(X\vee \mathbb{S}^m)$. So it is sufficient to show  $\text{TC}(X\vee \mathbb{S}^m)\leq \text{TC}(X)+1$. Indeed, by  (\cite{dranishnikov2014topological}, Theorem 3.6), $\text{TC}(X\vee \mathbb{S}^m)\leq \text{TC}^M(X\vee \mathbb{S}^m)\leq \text{TC}^M(X)+\text{TC}^M(\mathbb{S}^m)-1$. For $m$ odd, we note here $\text{TC}^M(\mathbb{S}^m)=\text{TC}(\mathbb{S}^m)=2$  and thus $\text{TC}(X\vee \mathbb{S}^m)\leq \text{TC}(X)+1$. 

Assume now that $m$ is even. Let us show that then $\text{TC}^M(X\vee \mathbb{S}^m)\leq \text{TC}^M(X)+1$. We will use Dranishnikov's construction from \cite{dranishnikov2014topological}. As shown in (\cite{farber2003topological}, Theorem 8), a reserved motion planning algorithm on $\mathbb{S}^m$  with open cover is given by:
\begin{eqnarray*}
\widetilde{V}_0 &=& \{(A,B)\in \mathbb{S}^m\times \mathbb{S}^m:A\neq -B\}\\
\widetilde{V}_1 &=& \{(A,B)\in \mathbb{S}^m\times \mathbb{S}^m:A\neq B \text{ and } B\neq B_0\}\\
\widetilde{V}_2 &=& (\mathbb{S}^m-C)\times (\mathbb{S}^m-C)
\end{eqnarray*} for some $B_0\in \mathbb{S}^m$ and $C\in \mathbb{S}^m$ with $C$ distinct from $B_0$ and $-B_0$. We note that $\widetilde{V}_2\cap (C\times \mathbb{S}^m)=\emptyset$ and $\widetilde{V}_2\cap (\mathbb{S}^m\times C)=\emptyset$. Let $\text{TC}^M(X)=n+1$. Then there is a reserved motion planning algorithm on $X\times X$ with open cover $\widetilde{U}_0,\ldots,\widetilde{U}_n$. By \cite{dranishnikov2014topological} there is an open $(n+1)-$cover $\widetilde{U}_0,\ldots,\widetilde{U}_n,\ldots,\widetilde{U}_{n+2}$ of $X\times X$ by sets strictly deformable to the diagonal and there are strict deformations \[D^k_X:\widetilde{U}_k\times I\to X\times X\] of $\widetilde{U}_k$ to $\Delta X$ that preserve the faces $X\times x_0$ and $x_0\times X$. Similarly, there is an open $3-$cover $\widetilde{V}_0,\ldots,\widetilde{V}_2,\ldots,\widetilde{V}_{n+2}$ of $\mathbb{S}^m\times \mathbb{S}^m$ and there are strict deformations $D^k_Y$ of $\widetilde{V}_k$ in $\mathbb{S}^m\times \mathbb{S}^m$ to the diagonal $\Delta \mathbb{S}^m$ that preserve faces $C\times \mathbb{S}^m$ and $\mathbb{S}^m\times C$. Set \[U_k=\widetilde{U}_k\cap (X\times x_0), ~~  V_k=\widetilde{V}_k\cap (C\times \mathbb{S}^m),\text{ for } k=0,\ldots,n+2.\] Note that $U_0,\ldots,U_{n+2}$ is an $(n+1)-$cover of $X\times x_0=X$ and $V_0,\ldots,V_{n+2}$ is a $3-$cover of $C\times \mathbb{S}^m=\mathbb{S}^m$. Let $W_k=U_k\times V_k$. By \cite{dranishnikov2014topological}, $W_0,\ldots,W_{n+2}$ is an open cover of $X\times \mathbb{S}^m$. Note that $W_2=\emptyset$, because $\widetilde{V}_2\cap (C\times \mathbb{S}^m)=\emptyset$. We eliminate the empty set $W_2$ and gave a new enumeration so that $W_0,\ldots,W_{n+1}$ is an open cover of $X\times \mathbb{S}^m$ with all the sets $W_k$ nonempty. 

By Symmetry, we can obtain an open cover $W^\prime_0,\ldots,W^\prime_{n+1}$ of $\mathbb{S}^m \times X$ with all the sets $W^\prime_k$ nonempty. Set \[O_k=W_k\cup W^\prime_k\cup\widetilde{U}_k\cup \widetilde{V}_k,\text{ for } k=0,\ldots n+1, \] and note that $\{O_k\}$ covers $(X\vee \mathbb{S}^m)\times (X\vee \mathbb{S}^m)$. Note that the set \[C=X\vee \mathbb{S}^m\vee X\vee \mathbb{S}^m \] defines a partition of $(X\vee \mathbb{S}^m)\times (X\vee \mathbb{S}^m)$ into four pieces: $X\times X, X\times \mathbb{S}^m, \mathbb{S}^m\times X,$ and $\mathbb{S}^m\times \mathbb{S}^m$. By Dranishnikov's construction, there are deformations $T_k:W_k\times I\to X\times \mathbb{S}^m$ and $T^\prime_k:W^\prime_k\times I\to \mathbb{S}^m\times X$ such that the deformations $D^k_X, D^k_Y, T_k,$ and $T^\prime_k$ all agree on $O_k\cap C$. Then union of the deformations \[Q_k=D^k_X\cup D^k_Y\cup T_k\cup T^\prime_k:O_k\times I\to (X\vee \mathbb{S}^m)\times (X\vee \mathbb{S}^m)\] is well-defined and each $Q_k$ defines a reserved section $\alpha_k:O_k\to P(X\vee\mathbb{S}^m)$. Therefore $\text{TC}^M(X\vee \mathbb{S}^m)\leq n+1+1= \text{TC}^M(X)+1$. 
\end{proof}

\begin{example}\label{tc-wedge-liegroup} We have $\text{TC}(\mathbb{RP}^3\vee \mathbb{S}^5)=5$. More generally, if $G$ is a compact, connected Lie group such that $\text{cat}(G\times \mathbb{S}^m)=\text{cat}(G)+1$. Then \[\text{TC}(G\vee \mathbb{S}^m)=\text{TC}(G)+1,\] because $\text{cat}(G)=\text{TC}(G)=\text{TC}^M(G)$.
\end{example}

On the other hand, it is easy to verify the following statement.

\begin{lemma}\label{tc-product}
Suppose that $X$ and $Y$ are path connected finite CW complexes such that $\text{TC}(X)=zcl_{\mathbb{K}}(X)+1$ and $\text{TC}(Y)=zcl_{\mathbb{K}}(Y)+1$, where $\mathbb{K}$ is a field.  Then
\[\text{TC}(X\times Y)=\text{TC}(X)+\text{TC}(Y)-1.\] Furthermore, $\text{TC}(X\times Y)=zcl_{\mathbb{K}}(X\times Y)+1$.
\end{lemma}

 The following theorem is the other principal result of this paper.
\begin{theorem}\label{theorem=tc}
 Let $\mathbb{K}$ be a field and $G$ an $m-$dimensional connected, compact Lie group ($m\geq 1$) and $n\geq 1$, such that $\text{cat}(G\times \mathbb{S}^{m+n-1})=\text{cat}(G)+1$(for example, if $G$ satisfies Ganea's conjecture) and $\text{TC}(G)=zcl_{\mathbb{K}}(G)+1$. Then
\begin{equation*}
\text{TC}(F(G\times\mathbb{R}^n,2))= 2\text{TC}(G)=2\text{cat}(G). 
\end{equation*}
Furthermore  $\text{TC}(F(G\times\mathbb{R}^n,2)) = zcl_{\mathbb{K}}(F(G\times\mathbb{R}^n,2))+1$.
\end{theorem}
\begin{proof}
Since $\text{TC}(X)$ is a homotopy invariant of $X$, it follows from Proposition \ref{homotopy-type} that $\text{TC}(F(G\times\mathbb{R}^n,2))=\text{TC}(G\times (G\vee \mathbb{S}^{m+n-1}))$. Next, $\text{TC}(G\vee \mathbb{S}^{m+n-1})=\text{TC}(G)+1$ by Example \ref{tc-wedge-liegroup}. By Lemma \ref{prop-cup-length}, $\text{TC}(G\vee \mathbb{S}^{m+n-1})=zcl_{\mathbb{K}}(G\vee \mathbb{S}^{m+n-1})+1$ and so, by Lemma \ref{tc-product}, $\text{TC}(F(G\times\mathbb{R}^n,2))= 2\text{TC}(G)$.
\end{proof}

\section{\bf Applications}

\subsection{Planar motion}

The space $SO(2)\times \mathbb{R}^2$, which is homeomorphic to $\mathbb{S}^1\times\mathbb{R}^2$, describes all planar motions of a rigid body in the $2-$dimensional space $\mathbb{R}^2$. Thus, we are interested in the configuration space of $k$ distinct points in the product $\mathbb{S}^1\times\mathbb{R}^2$ which describes all planar motions of $k$ robots in $\mathbb{R}^2$ and allows the bodies to occupy the same point in the plane (as long as their orientations are different).

The following Lemma generalises Lemma 10.2 given by Michael Farber in \cite{farber2004instabilities}.
\begin{lemma}\label{tc-wedge-spheres}
Let $Z=\underbrace{\mathbb{S}^{m_1}\vee\cdots\vee\mathbb{S}^{m_n}}_{n \text{ factors }}$ denote the wedge of spheres $\mathbb{S}^{m_i}$. Then,
\begin{equation*}
\text{TC}(Z)=\left\{
  \begin{array}{ll}
    2, & \hbox{if $n=1$ and $m_1$ is odd;} \\
    3, & \hbox{if either $n>1$ or some $m_i$ is even.}
  \end{array}
\right.
\end{equation*}
\end{lemma}
\begin{proof}
The case $n=1$ and $m_1$ is odd, is well known (see \cite{farber2003topological}, Theorem 8). On the other hand, \begin{eqnarray}
\text{TC}(Z) &\leq& 2\text{cat}(Z)-1 \label{ss}\\
      &=& 2\cdot 2-1\label{tt}\\
      &=& 3\nonumber, 
\end{eqnarray} where (\ref{ss}) follows from (\cite{farber2003topological}, Theorem 5) and (\ref{tt}) follows from Example \ref{cat-wedges-spheres}.

Case 1: some $m_i$ is even. We have \begin{eqnarray}
3 &=& \text{TC}(\mathbb{S}^{m_i}) \label{tcmi} \\
  &\leq& \text{TC}(Z) \label{tcs-tsz} \\ &\leq& 3\nonumber,
\end{eqnarray} where (\ref{tcmi}) follows from (\cite{farber2003topological}, Theorem 8) and (\ref{tcs-tsz}) follows from (\cite{dranishnikov2014topological}, Theorem 3.6). Thus $\text{TC}(Z)=3$.

Case 2: $n\geq 2$. We have,
\begin{eqnarray}
3 &=& \text{cat}(\mathbb{S}^{m_1}\times \mathbb{S}^{m_2})\nonumber\\
  &\leq& \text{TC}(Z)\label{des} \\
  &\leq& 3\nonumber,
\end{eqnarray} where (\ref{des}) follows from (\cite{dranishnikov2014topological}, Theorem 3.6). Thus $\text{TC}(Z)=3$.
\end{proof}

\begin{proposition}\label{planar}
For $n\geq 1$, 
\begin{eqnarray*}
 \text{TC}(F(\mathbb{S}^1\times \mathbb{R}^n,2)) &=& 4.
\end{eqnarray*} 
\end{proposition} 
\begin{proof} We give two proofs of the proposition. For the first, note that from Lemma \ref{tc-wedge-spheres},  $\text{TC}(\mathbb{S}^{m_1}\vee\mathbb{S}^{m_2})=zcl_{\mathbb{Q}}(\mathbb{S}^{m_1}\vee\mathbb{S}^{m_2})+1=3$. For the second use Theorem \ref{theorem=tc}.
 \end{proof}

 \begin{remark} There are two reasons for studying these particular configuration spaces.
 \begin{itemize}
     \item First, the ordered configuration space $F(\mathbb{S}^1\times \mathbb{R},2)$ describes the simultaneous movement  of two particles, without collisions, on the cylinder $\mathbb{S}^1\times \mathbb{R}$ (see left-hand side of  Figure \ref{figu}). The ordered configuration space $F(\mathbb{S}^1\times \mathbb{D}^2,2)$, which is homotopy equivalent to $F(\mathbb{S}^1\times \mathbb{R}^2,2)$, describes the simultaneous movement of two particles, without collisions, on the solid torus $\mathbb{S}^1\times \mathbb{D}^2$ (see right-hand side of Figure \ref{figu}).
 \begin{figure}[htb]
 \centering
 \caption{$F(\mathbb{S}^1\times \mathbb{R},2)$ and $F(\mathbb{S}^1\times \mathbb{D}^2,2)$}
 \label{figu}
 \includegraphics[scale=0.6]{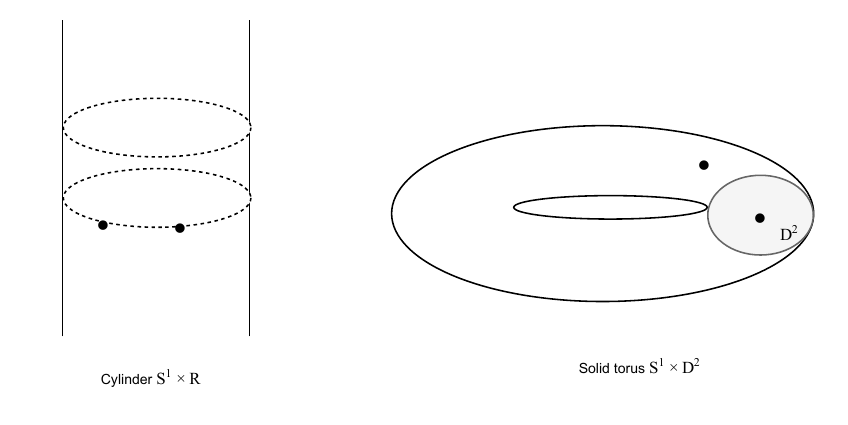}
\end{figure} 

\item Second, the space $\mathbb{S}^1\times \mathbb{R}^3$ describes movements of a rigid body, with a fixed point but under the influence of gravity, in the $3-$dimensional space $\mathbb{R}^3$. In this situation, the circle $\mathbb{S}^1$ consists of rotations about the direction of gravity. 
 \end{itemize}
 \end{remark}

\begin{remark}
 We can compare Proposition \ref{planar} with the topological complexity of the Cartesian product $(\mathbb{S}^1\times \mathbb{R}^n)^{2}$, which is homotopy equivalent to the product $\mathbb{S}^1\times\mathbb{S}^1$. By Lemma \ref{tc-product} (or \cite{farber2003topological}, Theorem 13) we easily obtain \begin{equation*}
 \text{TC}((\mathbb{S}^1\times \mathbb{R}^n)^{2}) = 3.
 \end{equation*} Thus, on $\mathbb{S}^1\times \mathbb{R}^n$, the complexity of the collision-free motion planning problem for two robots is more complicated than the complexity of the similar problem when the robots are allowed to collide.  
\end{remark}
 
\subsection{Spatial motion}

The space $SO(3)\times \mathbb{R}^3$, which is homeomorphic to $\mathbb{RP}^3\times\mathbb{R}^3$, describes all spatial motions of a rigid body in the $3-$dimensional space $\mathbb{R}^3$. Thus, we are interested in the configuration space of $k$ distinct points in the product $\mathbb{RP}^3\times\mathbb{R}^3$ which describes all spatial motions of $k$ robots in $\mathbb{R}^3$ and allows the bodies to occupy the same point in space (as long as their orientations are different).

\begin{proposition}\label{spatial}
For $n\geq 1$,
\begin{eqnarray*}
 \text{TC}(F(\mathbb{RP}^3\times\mathbb{R}^n,2)) &=& 8.
\end{eqnarray*}  
\end{proposition} 
\begin{proof}
 The proposition follows immediately from Theorem \ref{theorem=tc}.
\end{proof}

\begin{remark}
  We can compare Proposition \ref{spatial} with the topological complexity of the Cartesian product $(\mathbb{RP}^3\times\mathbb{R}^n)^{2}$, which is homotopy equivalent to the product $\mathbb{RP}^3\times\mathbb{RP}^3$. By Lemma \ref{tc-product},  \begin{equation*}
 \text{TC}((\mathbb{RP}^3\times\mathbb{R}^n)^{2}) = 7.
 \end{equation*} Thus, on $\mathbb{RP}^3\times\mathbb{RP}^n$, the complexity of the collision-free motion planning problem for two robots is more complicated than the complexity of the similar problem when the robots are allowed to collide.  
\end{remark}

\bibliographystyle{plain}

\end{document}